\documentclass[journal]{IEEEtran}

\usepackage[latin1]{inputenc}
\usepackage{amsmath}
\usepackage{authblk}
\usepackage{amsfonts}
\usepackage{amssymb}
\usepackage{amsthm}
\usepackage{indentfirst}
\usepackage[english]{babel}
\usepackage{graphicx}

\newtheorem{defn}{Definition}[section]
\newtheorem{thm}[defn]{Theorem}

\newtheorem{lemma}[defn]{Lemma}
\newtheorem{oss}[defn]{Remark}

\begin{document}


\title{\bf A new restart procedure for combinatorial optimization and its convergence}

\author{Davide Palmigiani and Giovanni Sebastiani
\thanks{Both authors are with Istituto per le Applicazioni del Calcolo ``Mauro Picone'', CNR, Rome, Italy, and with
Istituto ``Guido Castelnuovo'', ``Sapienza'' Universit{\`a} di Roma, Italy, email: palmigiani@mat.uniroma1.it, sebastia@mat.uniroma1.it}}
\maketitle


\begin{abstract}
We propose  a new iterative procedure to optimize the restart for
 meta-heuristic algorithms to solve combinatorial optimization, 
which uses independent algorithm executions.
The new procedure consists of either adding new executions or extending along time the existing ones. This is done on the basis  
of a criterion that uses a surrogate of the algorithm failure probability, 
  where the optimal solution is replaced by the 
\emph{best  so far} one.
 Therefore, it can be applied in practice.  
We prove that,  with probability one, the restart time of 
the proposed procedure
approaches, as the number of iterations diverges, the optimal value that minimizes 
the expected time to find the solution.
We apply the proposed restart procedure to several Traveling Salesman Problem instances with hundreds or thousands of cities. As basic algorithm, we used different versions of an Ant Colony Optimization algorithm.
We compare the results from the restart procedure with those from the basic algorithm.  
This is done  by considering the  failure probability of the two approaches for equal computation cost.
 This comparison  showed a significant gain when applying the proposed restart procedure, whose failure probability is several orders of magnitude lower.
\end{abstract}

\noindent


\section{Introduction}

Solving a combinatorial optimization problem (COP) consists of finding an element, within a finite search domain, which minimizes a given so called \emph{fitness function}.
The domain has typically a combinatorial nature, e.g.\ the space of the hamiltonian paths on a complete graph. 
The COP prototype is the Traveling Salesman Problem (TSP), whose solution is an  Hamiltonian cycle on a weighted graph with minimal total weight \cite{Applegate}. Although a  solution of a COP always exists, finding it may involve   a very  high computational cost.  The study of the computational cost of numerical algorithms  started in the early  1940s with the first introduction of computers.
Two different kinds of algorithms can be used to solve a COP problem: exact or heuristic. A method of the  former type  consists of   a sequence  of non-ambiguous and computable operations producing a COP solution in a finite time. Unfortunately, it is  often not possible to use exact algorithms.
This is the case for instances of  a $\mathcal{NP}$-complete COP.  In fact,  to establish with certainty  if any element of the search space is a solution, requires  non-polynomial computational cost.  Alternatively, heuristic  algorithms  can be applied. Such type of algorithms only guarantee either  a solution in an infinite time or a suboptimal solution. Of great importance are  the \emph{meta-heuristic} algorithms (MHA) \cite{BLUMROLI}. They are independent of  the particular COP considered, and often stochastic. Among them, there are Simulated Annealing \cite{SIMANN1},  Tabu Search \cite{GLO},  Genetic Algorithms \cite{GENALG}
 and Ant Colony Optimization (ACO) \cite{ACODES}. 

A natural issue for MHA concerns their convergence \cite{Geman2}, \cite{SCH1}, \cite{gutjahr03}, \cite{Neu06}, \cite{Gutjahr08}. Due to  the stochastic nature  of such algorithms, they have to be studied probabilistically; unfortunately, even when their convergence
is theoretically guaranteed, it is often too slow to  successfully use them in practice. One possible way to cope with this problem is  the so called \emph{restart} approach, which, aside from the present context, it is used more generally for simulating rare events \cite{GarvKroese98}, \cite{GarvKroese99}, 
\cite{Misev}. It consists  of several independent executions of  a given MHA:
the executions are randomly initialized and the best solution, among those produced,
is chosen. When implementing the restart on a non-parallel machine, the restart
consists of periodic re-initialitations of the underlying MHA, the period $T$ being called {\it restart time}.

Despite the fact that the restart approach is widely used, very little work has been done
to study it theoretically for combinatorial optimization \cite{H02}, \cite{vMW}. 
In \cite{H02}, the restart is studied in its dynamic form instead of the static one considered here. 
Some results are provided for a specific evolutionary algorithm, i.e.\ the so called (1+1)EA, used to minimize three pseuso-Boolean functions.  In \cite{vMW} the fixed restart strategy is considered as done here. The first two moments of 
the random time $T_R$ for the restart to find a solution (optimization time) are studied as a function of $T$. An equation for $T$ is derived, whose solution minimizes the
expected value of $T_R$. However, this equation involves the distribution of the optimization time of the underlying MHA, which is unknown.

In practice, the underlying MHA is very commonly restarted when there are negligible differences in the fitness of the best-so-far solutions at consecutive iterations during a certain time interval. This criterion may not be adequate when we want to really find the COP solution and we are not satisfied with suboptimal ones. 

The failure probability of the restart is
$$
\mathbb{P}(T_R>k)=p(T)^{\left\lfloor\frac{k-1}{T}\right\rfloor} p\left(k-\left\lfloor\frac{k-1}{T}\right\rfloor T\right)\, ,
$$
where $p(t) $ is the failure probability of the underlying MHA.
The restart failure probability is geometrically decreasing towards zero with the number $k$ of restarts, the base of such geometric sequence being $p(T)$.
Therefore, a short restart time may result in a  slow convergence.
On the contrary, if the restart time 
is high, we may end up with a low number  of restarts and a high value of the restart failure probability. 
Then, a  natural problem is to  find 
 an ``optimal value''  of $T$  when using a finite amount of computation time.

Following  \cite{CarvSeb11}, the restart could be optimized by 
choosing a value for $T$ that minimizes the expected value of the time $T_R$:
\begin{equation}
\mathbb{E}[T_R]=\sum_{k=1}^{\infty}\mathbb{P}(T_R>k)\, .\label{exptime}
\end{equation}
In fact for any random variable not negative  $X$ it is possible to write
\begin{equation}
\mathbb{E}[X]=\int_0^{\infty}\mathbb{P}(X>t)\, dt\, .\label{eqexp1}
\end{equation}
In our case, the random variable $T_R$ is discrete and the integral 
in (\ref{eqexp1}) is replaced by a series whose generic term is  $\mathbb{P}(T_R>t)$. 

We now derive an upper bound for the r.h.s.\ of (\ref{exptime}):

\begin{equation}
\mathbb{E}[T_R]\leq\sum_{k=1}^{\infty}
p(T)^{\left\lfloor\frac{k-1}{T}\right\rfloor} 
\leq\sum_{k=1}^{\infty}p(T)^{\frac{k-1}{T} -1} 
=\frac{1}{(1-p(T)^{\frac{1}{T}})p(T)}\, .\label{upperbound}
\end{equation}
By means of  this bound, we can then optimize the RP by minimizing the function
$g(x):=\left [(1-p(x)^{\frac{1}{x}})p(x)\right ]^{-1}\, .$

 Whenever this function is monotonically decreasing, there is  no advantage to use the restart. In the other case, an optimal value for the restart time
could be the first value $t_m$ where the function $g$ assumes 
its minimum.
However, this criterion cannot be applied  in practice  since the MHA failure probability is unknown.

Here we propose  a new iterative procedure to optimize  the restart. It does not rely on the MHA failure probability. Therefore, it can be applied in practice.  One procedure iteration consists of either adding new MHA executions or extending along time the existing ones.  Along the iterations, the  procedure  uses an estimate of the MHA failure probability   where the optimal solution is replaced by the \emph{best  so far} one.
We make  the hypothesis that the MHA failure probability 
converges to zero with the number of iterations. Then, we prove that,  with probability one, the restart time of the proposed procedure
approaches, as the number of iteration diverges, the optimal value $t_m$.
We also show the results of the application of  the proposed restart procedure to several TSP instances with hundreds or thousands of cities. As MHA we use different versions of an ACO.
Based on a large number of experiments, we compare the results from the restart procedure with those from the basic ACO. 
This is done  by considering the  failure probability of the two approaches  with the same total computation cost. The two algorithms have been implemented in MATLAB and C.  
 This comparison  shows a significant gain when applying the proposed restart procedure.

\section{The procedure}
\noindent The restart procedure (RP)  starts by executing $r_0$ independent replications of the underlying MHA for a certain number of time steps $T_0$.
Let us denote by $X_i(t)$ the solution produced by the replication $i$ of the underlying algorithm at time $t$.
 Then, at the end of iteration $k$, based on the criterion described later in this section, the RP  either increases  the number of replications from $r_k$ to $r_{k+1}$ by executing $r_{k+1}-r_k$  replications of the underlying algorithm until time $T_k$, or it continues the execution of the existing $r_k$ replications until time $T_{k+1}> T_k$. Let $Y_i(t)$ be the value of the best solution found by $i$-th replication until time $t$ i.e. $Y_i(t)=\min\{f(X_i(s)), s=1,...,t\}$, where $f$ is the function to minimize. Each $Y_i(t)$ is an  independent realization of the same process. We can always think  at the RP  in the following way: given the infinite matrix $\bf{Y}$ with generic element $Y_i(t)$ where $i,t=1,2,\dots$,  a realization of the RP produces, at each iteration $k$, a nested sequence $\{Y_{A_k}\}_{k\in\mathbb{N}}$ of  finite matrices, where $A_k:=\{(i,t):i=1,\dots,r_k\quad t=1,\dots ,T_k\}$. The matrix $Y_{A_k}$ corresponds to the first $r_k$ rows and $T_k$   columns of $\bf{Y}$.
Let $\tilde Y_k$ denote the minimum value of this matrix at the end of iteration $k$: $\displaystyle \tilde Y_k=\min{Y_{A_k}}=\min_{(i,t)\in A_k}Y_i(t)$. We  estimate   the failure probability sequence   by means of the empirical frequency 
 \begin{equation}\nonumber 
\hat p_k( t)=
\begin{cases}
 \displaystyle{\frac {1}{ r_k}}\sum_{i=1}^{r_k} 1_{\{Y_i(t)> \tilde Y_k\}}       & t=1,\dots,T_k, \\ 
0 & \text{otherwise.}
\end{cases}
\end{equation}
Next, consider the function  $g_k(t)=[(1-\hat p_k(t)^\frac 1 t)\hat p_k(t)]^{-1}$, $t=1,\ldots,T_k$, and 
define $\hat\sigma_k$ the first position with a left and right increase of the function $g_k$ (relative minimum).  Let $\lambda$ be a  number in $(0,1)$.  If $\displaystyle\hat\sigma_k<\lambda\cdot T_k$, then the RP increases the number of replications by means of a certain rule $r_{k+1}:=f_r(r_k)$. 
Otherwise, the RP increases the restart time  according to    $T_{k+1}:=f_T(T_k)$. We assume that  $\forall x$ we have $f_r(x)>x$  and $f_T(x)>x$.
As a consequence,  for any fixed $x>0$, it holds $f_r^{(k)}(x),f_T^{(k)}(x)\rightarrow\infty$, $k$ denoting  the consecutive application of a function for $k$ times. 
Therefore, the recursive formula for $(r_k,T_k)$ is $$(r_{k+1},T_{k+1})=\begin{cases}
  (f_r(r_k),T_k)& \text{if $\hat\sigma_k<\lambda\cdot T_k$,} \\ 
(r_k,f_T(T_k))& \text{otherwise.}
\end{cases}$$   

\noindent
Below there is the pseudo code for RP: 
\\
$r=r_0$;\\
$T=T_0$;\\
\textbf{for} replication $i=1,2,\dots,r$ \textbf{do} \\
\indent execute algorithm $\mathcal{A}$ until time $T_0$; \\
\indent save $\mathcal{A}_i(T_0)$;\\
\textbf{end for}\\
save $Y_{A_0}$;\\
compute $\hat\sigma_0$ from $Y_{A_0}$;\\
\textbf{for} iteration $k=1,2,\dots$ \textbf{do} \\
\indent\textbf{if} $\hat\sigma_{k-1}\ge\lambda\cdot T_{k-1}$ \textbf{then}\\
 \indent\indent $T_{k}=f_T(T_{k-1})$;\\
 \indent\indent $r_{k}=r_{k-1}$;\\
 \indent\indent\textbf{for} replication $i=1,2,\dots,r_k$ \textbf{do} \\
 \indent\indent\indent continue the execution of  $\mathcal{A}_i$ until $T_k$;\\
 \indent\indent\indent save $\mathcal{A}_i(T_k)$;\\
 \indent\indent\textbf{end for}\\
 \indent\textbf{else then}\\
 \indent\indent $r_k=f_r(r_{k-1})$;\\
 \indent\indent $T_{k}=T_{k-1}$;\\
\indent\indent\textbf{for} replication $i=r_{k-1}+1,r_{k-1}+2,\dots,r_k$ \textbf{do} \\
\indent\indent\indent execute $\mathcal{A}_i$ until $T_k$;\\
 \indent\indent\indent save $\mathcal{A}_i(T_k)$;\\
\indent\indent\textbf{end for}\\
\indent\textbf{end if}\\
\indent save $Y_{A_k}$;\\
\indent compute $\hat\sigma_k$ from $Y_{A_k}$;\\
\textbf{end for}\\

\section{RP convergence}
\noindent We denote by $f_m$ the value of the solution of the optimization problem. Moreover,
we recall the functions $g(t)=[(1-p(t)^\frac 1 t)p(t)]^{-1}$, where $p(t) $ is the failure probability and
$g_k(t)=[(1-\hat p_k(t)^\frac 1 t)\hat p_k(t)]^{-1}$, whose domain is $\{1,\ldots,T_k\}$.

In order to derive the following results,  we assume that  
\begin{enumerate}
\item $\displaystyle p(t)\underset{t\rightarrow\infty}\longrightarrow 0$,
\item $g(t)$ admits only one point of  minimum $t_m$ and it is   strictly decreasing for $t\leq t_m$,
\item $p(1)<1$.
\end{enumerate}
We notice that point $1$ just ensures that the underlying algorithm will eventually find the solution of the problem,
which is quite a natural requirement.
Moreover, point $3$ in practice does not give limitations. In fact, we can always aggregate some 
initial iterations of the algorithm into a single one. 
Because of point $1$, the aggregate iteration can always be chosen in such a way that point $3$ is satisfied. 

\begin{oss}\label{oss1}
We notice that,  by  the assumptions on the functions $f_r$ and   $f_T$, and by the RP  definition,  the probability that  both the sequences $r_k$ and $T_k$ are  bounded is zero.
\end{oss}
\begin{lemma}\label{lemmaconvprob}
Let $p(t)$ be as above. Let $(r_k,T_k)$ be  the  sequence of random variables which describes the RP. Then, it holds
\begin{enumerate}
\item $\mathbb{P}\left (r_k\rightarrow\infty\right)=1$, 
\item $\label{convopt} \mathbb{P}\left(\left\{\exists k:\quad \tilde Y_k=f_m\right\}\right)= 1.$
 \end{enumerate}
\end{lemma}

\begin{proof}
$\mathbf{1.}$ If $\mathbb{P}\left (r_k\rightarrow\infty\right)<1$, then, with positive probability, the following three conditions hold for   a certain positive integer $r$:
\begin{description}
\item[i)] $r_k=r$ eventually;
\item[ii)] $T_k$ diverges (for Remark \ref{oss1});
\item[iii)] $\hat\sigma_k\ge\lambda T_k$ eventually (from ii) and the definition of the RP).
\end{description}
By assumption $1$ and ii), with probability one, the underlying $r$ copies of the algorithm have all reached the optimum  after a certain time $t_0$. Therefore, it follows that, for all $h$ large enough, we have $\hat p_h(t)=0$ for $ t_0\le t\le T_h $. Hence, eventually $\hat\sigma_h$ does not change as $h$ increases,  which is a contradiction with iii). Therefore $\mathbb{P}\left (r_k\rightarrow\infty\right)=1$.

$\mathbf{2.}$ 
By i), since $p(1)<1$, with probability one, there exists $i$ such that $Y_i(1)=f_m$; for all $k$ so large that $r_k\ge i$ it will be $\tilde Y_k=f_m$.
This proves the point.
\end{proof}

\begin{lemma}\label{lemma3.3}
For each $t\in \mathbb{N}$, it holds
\begin{equation}\label{eq1thm}
\mathbb{P}\left (\left\{\sup_{k}T_k<t\right\}\cup\left\{\sup_{k}T_k\ge t \, , \, \lim_{k\rightarrow\infty}\hat p_k(t)=p(t)\right\}\right )=1.\nonumber
\end{equation}
 \end{lemma}
\begin{proof}

Let us consider the case when the event  $E_t:=\left\{\sup_{k}T_k\ge t\right\}$ happens,  then we can eventually compute 
$\hat p_k(s)$,  for $s=1,\ldots, t$.
 By point $1$ of   Lemma \ref{lemmaconvprob} and  the strong law of large numbers, we get 
\begin{equation}\label{slln}
\mathbb{P}\left (E_k \, , \,   \lim_{k\rightarrow\infty}\frac 1 {r_k}\sum_{i=1}^{r_k}1_{\{Y_i(t)>f_m\}}=p(t)\right )
=\mathbb{P}\left (E_k \right ).\nonumber
\end{equation}
Hence, using point $2$ of Lemma \ref{lemmaconvprob}, we obtain
that
$$
\mathbb{P}\left (E_k \, , \,   \lim_{k\rightarrow\infty}\frac 1 {r_k}\sum_{i=1}^{r_k}1_{\{Y_i(t)>f_m\}}=p(t)\right ) $$
is equal to
$$
\mathbb{P}\left (E_k \, , \,   \lim_{k\rightarrow\infty}\frac 1 {r_k}\sum_{i=1}^{r_k}1_{\{Y_i(t)>\tilde Y_k\}}=p(t)\right ).
$$
Since, by definition 
$$
\hat p_k(t)=\frac 1 {r_k}\sum_{i=1}^{r_k}1_{\{Y_i(t)>\tilde Y_k\}},
$$
we have
\begin{equation}
\mathbb{P}\left (E_k \, , \,   \lim_{k\rightarrow\infty}\hat p_k(t) =p(t)\right )=
\mathbb{P}\left (E_k \right ),\nonumber
\end{equation}
from which the thesis follows.
\end{proof}

\begin{thm}\label{lemmatempo}
For the RP it holds
$$\mathbb{P}\left (\sup_k T_k>\frac{t_m}{\lambda}\right)=1.$$ 
\end{thm}
\begin{proof}
Let us assume that the thesis is not true.  Then, there exists  an integer number $M$ such that $M\le \frac{t_m}{\lambda}$ and $\mathbb{P}\left (\left \{\sup_k T_k=M\right\}\right)>0$. 
On the event $\left\{\sup_k T_k=M\right\}$, by  both Lemma (\ref{lemma3.3}) 
and the continuous mapping, we have   the convergence $g_k(t)\rightarrow g(t)$, for any $t\le M$.
This means that, for any $\varepsilon>0$  there is a positive probability  that 
$\displaystyle\bigcap_{t=1}^{M}\left\{\left\vert g_k(t)-g(t)\right\vert <\varepsilon\right\}$, when $k$ is large enough. 
Let us define $\tilde M:=\min(M,t_m)$.
We then have  $g_k(\tilde M)<g(\tilde M)+\varepsilon$ and  $ g_k(t)>g(t)-\varepsilon$ for any $1\leq t < \tilde M$.  
Subtracting the first inequality from the last one, we obtain
 $$g_k(t)-g_k(\tilde M)>g(t)-g(\tilde M)-2\varepsilon.$$ 
 Since $g$ is strictly decreasing until $t_m$, the r.h.s.\ of the last inequality is strictly larger than $g(\tilde M-1)-g(\tilde M)-2\epsilon$. 
By taking  $\varepsilon$ sufficiently small,
 we get $g_k(t)-g_k(\tilde M)>0$ for any $t<\tilde M$. 
Hence,  with  a positive probability,  we get  eventually  $\displaystyle\hat\sigma_k\ge \tilde M$. 

If $M\le t_m$,  then $\tilde M=M$ and with positive probability  eventually we have $\displaystyle\hat\sigma_k\ge M$. Since $\hat\sigma_k\le T_k\le\sup_k T_k=M$, we have eventually $\hat \sigma_k=T_k=M$. For one of such $k$, it holds $\displaystyle\frac{\hat\sigma_k}{T_k}=1>\lambda$, so that,  by the definition of RP, at the following iteration with positive probability we have $T_{k+1}>T_k=M=\sup_k T_k$, which is impossible. 

In the other case, $\displaystyle t_m<M\le\frac{t_m}{\lambda}$, we have $\tilde M=t_m$, and 
 there is a positive probability that  $t_m=\tilde M\le\hat\sigma_k\le T_k\le\sup_k T_k=M$ for $k$ large enough;  for any of these values of $k$, we get $ \displaystyle
\frac{\hat \sigma_k}{T_k}\ge
\frac{t_m}{T_k}\ge  \frac{t_m}{M}\ge \lambda$.
As a consequence, with positive probability eventually we have $T_{k+1}=f_T(T_k)>T_k$, which is a contraddiction with $\sup_k T_k=M$.
\end{proof}

\begin{thm}
  If we define $\displaystyle T:=\left\lceil\frac{ t_m}{\lambda}\right\rceil$, 
it holds 
\begin{enumerate}
\item 
 $\displaystyle \mathbb{P}\left (\bigcap_{t=1}^{T}\lim_{k\rightarrow\infty}\hat p_k(t)=p(t)\right )=1,$
\item 
$\displaystyle \mathbb{P}\left(\lim_{k\rightarrow\infty}\hat \sigma_k =t_m\right)=1.$
\end{enumerate}
\end{thm}

\begin{proof}
$\mathbf{1.}$  By theorem \ref{lemmatempo}, for any $t=1,2,\dots,T$, with probability one we can eventually compute $\hat p_k(t)$. Hence,
 by the two statements of   Lemma \ref{lemmaconvprob} and  the strong law of large numbers,  we get
$$\mathbb{P}\left(\bigcap_{t=1}^{T}\lim_{k\rightarrow\infty}\hat p_k(t)=p(t)\right)=1,$$ that completes the proof of this point.

$\mathbf{2.}$ 
From $1$ and the continuous mapping, with probability one, it holds
$$
\lim_{k\rightarrow\infty}g_k(t)=g(t),
$$
for $t=1,\ldots,T$, with $T> t_m$.
Therefore, the sequence $\hat \sigma_k$ converges to $t_m$.
\end{proof}

\begin{oss}
The efficiency of the RP  depends on the expected value of the ratio $\sup_k T_k/t_m$. Although we do not have
derived upper-bounds for this ratio, in all applications we performed, it remains sufficiently
close to one.  
\end{oss}

  \section{Numerical results}

\noindent 
   Below, we describe some results of the application of the RP  to  different 
 instances of the TSP studied in \cite{Stutzle00}  and to one pseudo-Boolean problem.
The underlying algorithm used here in the RP is  the ACO proposed in \cite{Stutzle00}, known as MMAS; for the TSP instances, it is combined with different local search procedures.
   The RP setting is as follows:
$r_{k+1}=f_r(r_k):=c_1\cdot r_{k}$ and
    $T_{k+1}= f_T(T_k):=c_2\cdot T_{k} $ where 
 $c_1=1.2, c_2=1.1$. The initial values for $r$ and $T$ are $20$ and $100$, respectively.
 Finally, we  set $\lambda=\frac 4 5$.   

For both the TSP instances  and the  pseudo-Boolean problem
considered here, the optimal solution is  known. 
This information can be used to estimate the failure probability of the RP and of the underlying algorithm. However, obviously this information cannot be used when applying the RP. 

In order to compare the results from the two algorithms with the same computational effort, we consider for  the RP a pseudo-time $t$, defined as follows:  
for the first RP iteration, the first $T_1$ instants of the  pseudo-time correspond to  the first $T_1$ iterations  of the first replication of the underlying algorithm; the following $T_1$ pseudo-time instants correspond to the analogous of the second replication and so on.  At the end of the $k$-th RP iteration, we  have produced $r_k$ executions  (replications) for  $T_k$ times and the final pseudo-time  instant is  $t=r_k\cdot T_k$.  At the $(k+1)$-th iteration,  we have a certain  $(r_{k+1},T_{k+1})$, with either $r_{k+1}>r_k$ and $T_{k+1}=T_k$ or $r_{k+1}=r_k$ and $T_{k+1}>T_k$. In the first case, the pseudo-time instant $t=T_k\cdot r_{k}+1$ corresponds to the first iteration  time  of the $r_{k}+1$ replication and it is increased  until  the end of that replication. We proceed in the same way  until the end of $r_{k+1}$ replication. In the second case, the pseudo-time instant $t=T_k\cdot r_k +1$ corresponds to the iteration time $T_k+1$ of the first replication and is then increased until the iteration time $T_{k+1}$  of that replication. Then, the same procedure is applied for the remaining replications based on their number.

We denote by  $\tilde Y(t)$ ($t=1,2,\dots$) the process describing   the best so far solution  of the RP (MMAS) corresponding to the pseudo-time (time) instant $t$. Hence, based on a set of $m$ replications of the RP, we can  estimate the failure probability   $ p_{\text{\tiny RP \normalsize}}\!\!(t)$
 by using the classical estimator 
\begin{equation}\label{classestim}
\hat p_{\text{\tiny RP \normalsize}}\!\!(t)=\frac {1} {m}\sum_{i=1}^{ m} 1_{\{\tilde Y_i(t)\neq f_m\}}\, ,
\end{equation}
and analogously with $\hat p(t)$ for the MMAS.
By the law of large numbers  this estimator converges to the failure  probability  $p_{\text{\tiny RP \normalsize}}\!\!(t)$  (to $p(t)$ for the MMAS).


   

We start with the example where we want to minimize the following pseudo-Boolean  function 
\begin{equation}\label{modrapp}
f(x)=-\left\vert\sum_{i=1}^N x_i-\frac{N-1}{2}\right\vert\, ,
\end{equation}
with respect to all binary strings of length $N$.
In Fig.\ \ref{functionboolean}, this function  is plotted versus the number of $1$s
in the case of $N=50$ considered.
\begin{figure}[t!]
\centering
\includegraphics[width=8.cm]{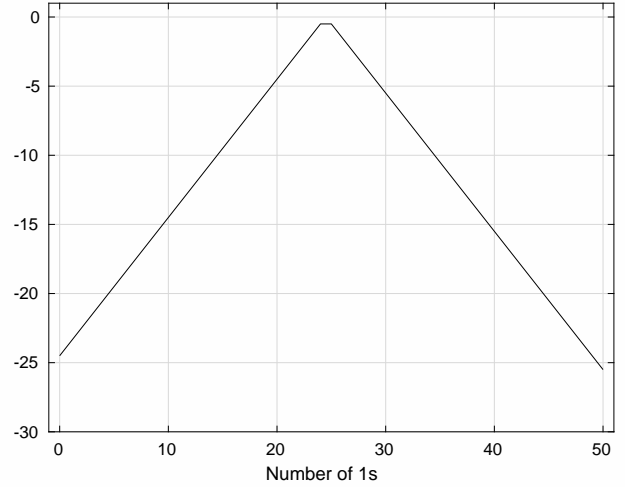}
\caption{\small Plot of the considered pseudo-Boolean function value versus the  number of $1s$ of the binary string.\label{functionboolean}}
\end{figure}
\begin{figure}[h!]
\centering
\includegraphics[width=8.cm]{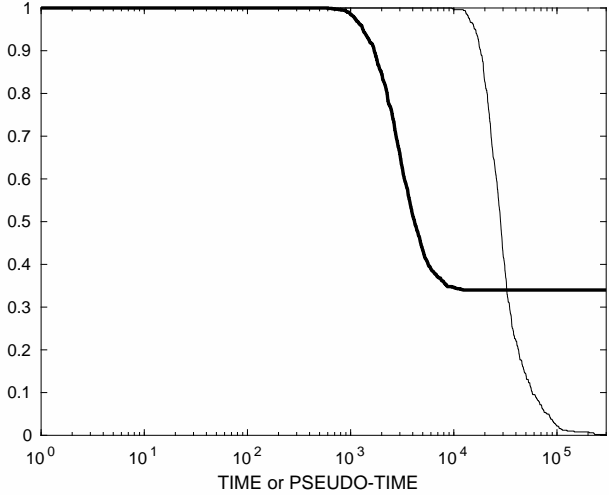}
\caption{\small Pseudo-Boolean problem. The estimated failure probability as function of time or pseudo-time for the standard MMAS (tick line) and the RP (thin line).The time axis is in logarithmic scale.
The f.p.\ curves of both the   RP and  the underlying algorithm     are computed by the estimator in (\ref{classestim})  based on $500$ and $1000$ replications, respectively.
\label{curveboolean}}
\end{figure}
This function has two local minima but only one of them is global. 
As written above, the ACO algorithm considered is the MMAS,
for which the pheromone bounds $ \tau_ {min} $ and $ \tau_{\max}$ ensure that
at any time, there is a positive probability to visit each configuration, e.g.\ the global minimum. Therefore, with probability one this algorithm will find the solution.
However, if it reaches a configuration with  few $1$s,
it takes in average an
enormous amount of time, not available in practice, to move towards the global minimum.
Therefore, we expect that in this case the restart will be successful.


In Fig.\ \ref{curveboolean}, we show the estimated failure probability (f.p.) $\hat p(t)$ for the MMAS algorithm to minimize the pseudo-boolean function of Fig.\ 
\ref{functionboolean} (tick line). In the same figure,  the 
estimated f.p.\ $ \hat p_{\text{\tiny RP \normalsize}}\!\!(t)$ of the RP is plotted versus the pseudo-time (thin line). 
We notice that there is a clear advantage to use the RP  when compared to the standard MMAS.

We consider now  an instance of the TSP with $532$ cities (att532). 
After five hundreds of thousands of iterations,
the underlying algorithm   has an estimated f.p.\ of $0.38$ ca. Instead, at the same value of the pseudo-time, the RP  has a significantly lower f.p.\ ($0.004$ ca), as clearly shown in Fig.\ \ref{532}. We remark that, until the value $3900$ ca for  the time or pseudo-time, the f.p.\ of the underlying algorithm  is lower  than the one of RP.  This is due to the fact that the RP   is still learning the optimal value of the restart time. After that, the trend is inverted: the RP overcomes the MMAS and gains
two  orders of magnitude for very large values of the pseudo-time. 

  \begin{figure}[h!]
\centering
\includegraphics[width=8.cm]{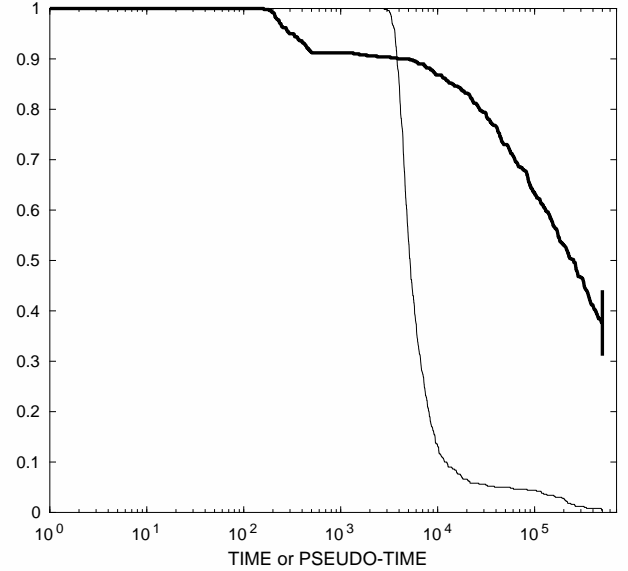}
\caption{\small The TSP instance with $532$ cities (att532). The estimated failure probability as function of time or pseudo-time for the standard MMAS (tick line) and the RP (thin line). The time axis is in logarithmic scale.
The f.p.\ curves of both the   RP and  the underlying algorithm    are computed by the estimator in (\ref{classestim})  based both on $500$  replications. The vertical segment shows  the  $99\%$ level confidence interval.\label{532}}
\end{figure}

We notice that the value $\hat\sigma_k$  approaches the optimal restart time $t_m$. In fact, as an example, in Fig.\ \ref{532sigma}, we show 
the denominator of the function $g_k(t)$ at the end
of a single RP execution. A global maximum appears at approximately the value of $430$, the difference with the value of $t_m$, computed from the estimate $\hat p(t)$, being less than $1 \%$. 


  \begin{figure}[h!]
\centering
\includegraphics[width=8.cm]{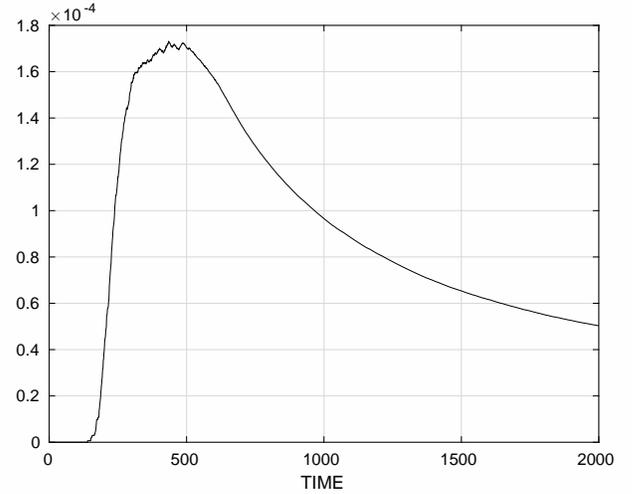}
\caption{\small The TSP instance with $532$ cities (att532).
The denominator of the function $g_k(t)$ at the end of a single RP replication.
 \label{532sigma}}
\end{figure}

 Finally in Fig.\  \ref{532restartfisso}, 
we compare the  f.p.\ curve for the RP 
with  the one obtained applying the restart periodically with the estimated  optimal restart time. We notice that this estimation requires much longer computation than to execute the RP. We remark that  the RP curve starts to decrease significantly
after the other one. This is due to the fact that the RP is still searching for the optimal value of the restart, whereas it is set from the beginning in the other (ideal) case. At about pseudo-time $7000$, the two f.p.s  become almost equal. 
After that, the f.p.\ of the MMAS goes to zero faster, even if the difference
between the two f.p.s remains less than $0.05$ ca. Finally, at pseudo-time $5\cdot 10^5$, the f.p.\ of the RP is $4\cdot 10^{-3}$.

We notice that curves similar to those as in Fig.\ \ref{532},  \ref{532sigma}  and  \ref{532restartfisso} were obtained
for all the other TSP instances considered. The relative results are shown in Table \ref{tabella}.

   \begin{figure}[h!]
\centering
\includegraphics[width=8.cm]{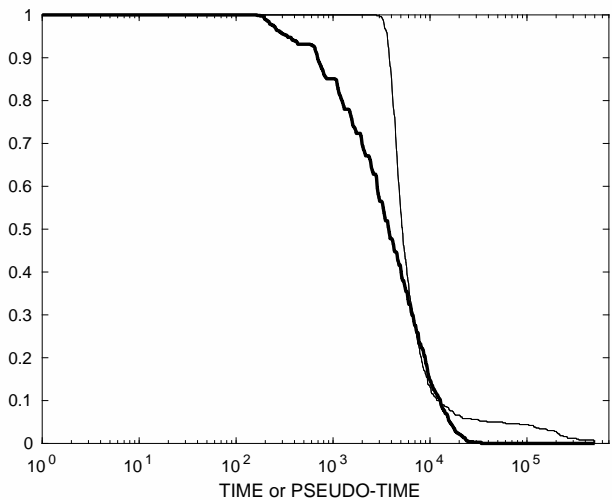}
\caption{\small The TSP instance with $532$ cities (att532).
Comparison between the  failure probability curve of the RP that appears in 
Fig.\ \ref{532} (thin line) and the one obtained applying the restart periodically with the optimal restart time (tick line).
The f.p.\ curves of both the   RP and  the underlying algorithm     are computed by the estimator in (\ref{classestim})  based on $500$ replications.
 \label{532restartfisso}}
\end{figure}
\begin{table}[h!]
\renewcommand{\arraystretch}{1.5}
\begin{tabular}{|*{5}{c|}|}
\hline
\bf{Instance}& \bf{ACO algorithm}&$\mathbf{T_c}$&\bf{ACO f.p. }&\bf{RP f.p.}\\
\hline
boolean50&MMAS&$300000$&$0.34$&$2.1\cdot 10^{-3}$\\
\hline
pcb442&MMAS-3opt&$100000$&$0.22$&$4.0\cdot 10^{-3}$\\
\hline
att532&MMAS-3opt&$500000$&$0.38$&$4.0\cdot 10^{-3}$\\
\hline
lin318&MMAS-2.5opt&$30000$&$0.44$&$0$\\
\hline
d1291&MMAS-3opt&$700000$&$0.57$&$2.0\cdot 10^{-3}$\\
\hline
d198&MMAS-2.5opt&$100000$&$0.67$&$0$\\
\hline
\hline
\end{tabular}
\vspace{.4cm}
\caption{Results of the application of  the RP and the underlying algorithm  to TSP instances with known optimal solutions and to the pseudo-Boolean problem.
  The failure probability (f.p.) values  are computed at the time $T_c$ reported  in the third column  (pseudo-time for the RP). Samples with at least $500$ elements are used.}
\label{tabella}
\end{table}      
  By looking at the results in Table \ref{tabella}, it is evident  the advantage of using the RP instead of the underlying algorithm. In fact, for all instances, the f.p.\ of the RP is several  orders of magnitude lower than the one of the underlying algorithm.

\section{Conclusions}
Given a combinatorial optimization problem, it is often needed  to apply stochastic  algorithms exploring the space using a general criterion independent of the problem. Unfortunately, usually there is a positive probability that the algorithm  remains  in a sub-optimal solution. This problem can be afforded by applying periodic algorithm
re-initializations. This strategy is called restart. Although  it  is often  applied  in practice, there are few works studying it theoretically.  In particular, there are no theoretical information about how to choose a convenient value for the restart time.

In this paper, we propose a new procedure to optimize the restart and we study it theoretically.
The iterative procedure  starts by executing  a certain number of replications
of the underlying algorithm for a predefined time. Then, at any following iteration $k$ of the RP, we compute the minimum value $\tilde Y_k$ of  the objective function. Hence, for each time $t=1,\ldots,T_k$, we estimate the failure probability $\hat p_k(t)$  that we have not yet reached the value  $\tilde Y_k$. After that, we compute 
the position $\hat\sigma_k$ of the first minimum of $g_k(t)$, which is a function  
of the failure probability.   
 If $\hat\sigma_k$ is close to the end of the current execution time frame of the underlying algorithm $T_k$, this last is increased; otherwise the number of replications is increased, which  improves the estimate $g_k(t)$ of $g(t)$.
This is controlled by the parameter $\lambda \in (0,1)$.
The position of the minimum of $g(t)$ corresponds to an ``optimal value'' of the restart time, that minimizes the expected  time to find a solution.
 
The theory predicts that the algorithm will find the optimal value of the restart. In fact, the theorems  proved  demonstrate that,  if $p(t)$ tends to zero, $g(t)$  has only one minimum  at position $t_m$ and it
is a strictly  decreasing function until $t_m$,  then, with probability one, $\hat p_k(t)$, $g_k(t)$ and its first minimum converge to $p(t)$,  $g(t)$ and $t_m$, respectively. 

In this paper, we have shown some results obtained by applying  the RP to several TSP instances with hundreds or thousands of cities. 
The results obtained have shown that the f.p.\ of the RP is several orders of magnitude lower than the one of the underlying algorithm, for equal computational cost. 
Therefore, given a certain computation resource, by applying the RP, we are far more confident that the result obtained  is a solution of the COP instance analyzed.    
The procedure proposed could be improved preserving its performance and decreasing the computational cost. 
A possible way to do it is to increase the parameter $\lambda$ along iterations. In fact, once we have a reasonably good estimate of $g(t)$, we would like
to reduce the possibility that, by chance, we increase too much the time interval length. This can be done by increasing the value of $\lambda$.

\section*{Acknowledgments}
\noindent
The authors are very thankful
to Prof.\ Mauro Piccioni for his very useful comments and suggestions and to
Prof.\ Thomas St\"{u}tzle for the ACOTSP code.

\bibliographystyle{IEEEtran}
\bibliography{restartIEEE}

\end{document}